\documentclass[twoside]{article}
\usepackage[preprint]{fourmodes}
\usepackage[authoryear,round]{natbib}
\newif\ifjournalformat
\journalformatfalse
\usepackage{amsmath,amssymb,amsthm}
\usepackage{graphicx,booktabs,etoolbox}
\usepackage{microtype}
\usepackage{hyperref}
\hypersetup{colorlinks=true,linkcolor=blue,citecolor=blue,urlcolor=blue,
 pdftitle={Three Homoscedastic Gaussian Components Have at Most Four Modes},
 pdfauthor={},pdfsubject={Anonymous manuscript for peer review}}
\theoremstyle{plain}
\newtheorem{theorem}{Theorem}
\newtheorem{lemma}[theorem]{Lemma}
\newtheorem{proposition}[theorem]{Proposition}
\newtheorem{corollary}[theorem]{Corollary}
\theoremstyle{definition}

\theoremstyle{remark}

\newcommand{\figalt}[1]{\ifjournalformat\par\smallskip{\footnotesize\textit{Alt text:} #1}\fi}
\fancypagestyle{arxivfirst}{\fancyhf{}\fancyfoot[C]{\thepage}\renewcommand{\headrulewidth}{0pt}\renewcommand{\footrulewidth}{0pt}}

\hypersetup{
  pdftitle={Three Homoscedastic Gaussian Components Have at Most Four Modes},
  pdfauthor={Rintaro Okahara; Kei Kobayashi},
  pdfsubject={Preprint}
}

\begin{document}
\runningtitle{Three Gaussian Components Have at Most Four Modes}
\runningauthor{Rintaro Okahara, Kei Kobayashi}
\twocolumn[
\aistatstitle{Three Homoscedastic Gaussian Components Have at Most Four Modes}
\vspace{-0.18in}
\begin{center}
\renewcommand{\arraystretch}{1.15}
\begin{tabular}{@{}c@{\hspace{3.7cm}}c@{}}
\textbf{Rintaro Okahara}$^{\ast}$ & \textbf{Kei Kobayashi} \\
The University of Tokyo & Keio University \\
\small\href{mailto:karin1129@g.ecc.u-tokyo.ac.jp}{\texttt{karin1129@g.ecc.u-tokyo.ac.jp}} &
\small\href{mailto:kei@math.keio.ac.jp}{\texttt{kei@math.keio.ac.jp}} \\
\footnotesize $^{\ast}$Corresponding author &
\end{tabular}
\end{center}
\vspace{0.18in}
]
\thispagestyle{arxivfirst}
\begin{abstract}
Gaussian mixtures can have more modes than components. For three Gaussian components with a common covariance, four modes are attainable, whereas the best general upper bound was eight. We prove that four is the exact maximum in every dimension $d \ge 2$, for arbitrary means and positive weights. We further show that the sharp bound drops to three when the whitened means are collinear or form a right or obtuse triangle. The result covers degenerate modes without genericity or separate finiteness assumptions. Our proof maps modes to local minima of a separable free-energy functional on the probability simplex, reduces the problem to counting scalar inverse-function branches, and controls competing branches through curvature and stationary-height comparisons.
 \end{abstract}

\section{Introduction}

A Gaussian mixture can have more modes than components.  Already three
isotropic Gaussian densities in the plane can produce four modes
\citep{carreirawilliams2003}.  The phenomenon makes the number of modes a
geometric question rather than a direct measure of the number of mixture
components.  It also leaves a basic extremal problem: how many modes can
three components produce when their covariance matrices are equal?

The general problem of bounding the number of modes of a Gaussian mixture
has attracted attention in statistics, geometry, and algebra
\citep{amendola2020}.  For three components with a common covariance,
Okuno and Kabata obtained the upper bound eight and formulated four as the
conjectured sharp bound \citep{okuno2026}.  We prove that bound, for arbitrary
positive weights and arbitrary means.  We also show that a fourth mode
requires an acute triangle after whitening the covariance.

Throughout, a mode is a local maximum point, including a degenerate local
maximum. For a symmetric positive-definite matrix $\Sigma\in\mathbb R^{d\times d}$,
write
\[
 \varphi_\Sigma(z)=\frac{1}{(2\pi)^{d/2}\sqrt{\det\Sigma}}
 \exp\!\left(-\frac12z^{\mathsf T}\Sigma^{-1}z\right).
\]

\begin{theorem}\label{thm:main}
Let
\[
 f(x)=\sum_{i=1}^3 w_i\varphi_\Sigma(x-\mu_i),
 \qquad w_i>0,\quad \sum_{i=1}^3 w_i=1,
\]
be a density on $\mathbb R^2$.  Then $f$ has at most four modes, and this
bound is sharp.  If the whitened means $\Sigma^{-1/2}\mu_i$ form an obtuse
or right triangle, or are collinear, then $f$ has at most three modes.
The bound three is sharp in each of these three classes.
\end{theorem}

\begin{corollary}[Dimension reduction]\label{cor:dimension}
For every $d\geq2$, a mixture of three Gaussian densities on
$\mathbb R^d$ with a common positive definite covariance has at most four
modes, and four is attained. The sharp bound is three when the whitened
means are collinear or form a right or obtuse triangle.
\end{corollary}
The proof, given in Section~\ref{sec:sharpness}, reduces the mixture to the
affine span of its three whitened means.

The proof uses a separable free-energy functional on a two-dimensional
simplex. It is obtained by profiling the location variable out of the
variational free energy used for latent-variable models
\citep{nealhinton1998}. That joint location--responsibility free energy is
written explicitly for Gaussian mean shift by \citet{carreira2006}.
The same entropy--quadratic structure is classical in ternary mean-field
and regular-solution models; an early global three-component phase diagram
was given by \citet{furman1977}. Its common-coefficient, uniform-weight specialization is the
zero-field Curie--Weiss--Potts free energy \citep{elliswang1990}; arbitrary
positive weights correspond to vector-valued external fields, whose global
phase diagram and sharp bound four were established by \citet{kulske2020}.
\citet{gaite1998} classify the local singularities of a more general
ternary family. In mixture geometry, a common quadratic coefficient
corresponds to equilateral whitened means. We allow unequal coefficients
and arbitrary positive weights; an obtuse triangle additionally gives one
negative coefficient, subject to a geometric constraint. Our contribution
is the uniform sharp count for this general geometry, not the underlying
variational free-energy identity or the four-mode construction.

\paragraph{Proof roadmap.}
After whitening, Section~\ref{sec:reduction} profiles out the location
variable and identifies modes with local minima of a separable free energy
$F$ on the probability simplex. Appendix~\ref{app:reduction} proves this
correspondence, including at degenerate stationary points and for affinely
dependent means. Its stationary equations split into scalar branches, so
mode counting becomes the problem of counting where a small number of branch
sums cross one. Sections~\ref{sec:curvature}
and~\ref{sec:folds} bound each branch type. Those separate bounds still permit
one configuration with four minima among the first three types;
Section~\ref{sec:comparison} rules out that configuration by comparing the
heights of the required extrema. A final acute-only branch contributes at
most one more minimum.
Appendix~\ref{app:genericity} removes genericity, and
Section~\ref{sec:sharpness} proves sharpness and the corollary above.
Appendices~\ref{app:reduction}--\ref{app:algebra} contain the detailed
calculations and perturbation arguments. We use only the finiteness result
of \citet[Theorem~1]{okuno2026}.

\section{A separable free energy and its stationary branches}
\label{sec:reduction}

This section first explains why the free energy $F$ is the useful object,
and then rewrites its minima as a one-dimensional branch-counting problem.
For the generic parameters used in that count, the all-lower branch
contributes at most one minimum and every other minimum is a downward
crossing of a scalar branch sum through one.

\paragraph{Whitening and notation.}
Let $\Sigma^{1/2}$ be the symmetric positive-definite square root of
$\Sigma$, and set
\[
 y=\Sigma^{-1/2}x,\qquad \nu_i=\Sigma^{-1/2}\mu_i,
\]
and define the transformed density
\begin{equation}\label{eq:whitening}
 \widetilde f(y):=\sqrt{\det\Sigma}\,f(\Sigma^{1/2}y)
 =\sum_{i=1}^3w_i\varphi_{I_2}(y-\nu_i).
\end{equation}
The map $x\leftrightarrow y$ is invertible and the prefactor is positive,
so $f$ and $\widetilde f$ have corresponding local maxima. Moreover,
\[
 \|\nu_i-\nu_j\|^2
 = (\mu_i-\mu_j)^{\mathsf T}\Sigma^{-1}(\mu_i-\mu_j),
\]
so Euclidean geometry after whitening is precisely the Mahalanobis geometry
of the original means. From now on we rename $y$ as $x$, $\nu_i$ as
$\mu_i$, and $\widetilde f$ as $f$. Thus all means, norms, distances, and
angles below are in whitened coordinates, and the common covariance is
$I_2$.

\paragraph{Why this free energy?}
Rather than count maxima of $f$ directly, we introduce posterior weights
$p$ and profile out $x$; this turns the problem into constrained
minimum counting on a simplex.
On the open simplex
\[
 \mathcal P=\{p\in(0,1)^3:p_1+p_2+p_3=1\},
\]
write $\bar\mu(p)=\sum_i p_i\mu_i$ and introduce the joint variational
energy
\[
 \begin{aligned}
 D_{\mathrm{KL}}(p\Vert w)&=\sum_i p_i\log\frac{p_i}{w_i},\\
 \mathcal E(x,p)&=D_{\mathrm{KL}}(p\Vert w)
     +\frac12\sum_i p_i\|x-\mu_i\|^2.
 \end{aligned}
\]
Completing the square in $x$ gives
\begin{equation}\label{eq:profiled-potential}
 \begin{aligned}
 \mathcal E(x,p)&=F(p)+\frac12\|x-\bar\mu(p)\|^2,\\
 F(p)&=D_{\mathrm{KL}}(p\Vert w)
      +\frac12\sum_{i<j}p_ip_jd_{ij}^2.
 \end{aligned}
\end{equation}
where $d_{ij}^2=\|\mu_i-\mu_j\|^2$. Hence
$F(p)=\min_x\mathcal E(x,p)$: it is the energy left after profiling out
the location. For fixed $x$, the Gibbs variational identity gives
\[
 -\log\!\bigl(2\pi f(x)\bigr)=\min_{p\in\mathcal P}\mathcal E(x,p),
\]
whose unique minimizer is the posterior probability vector. The two
minimizations meet at
\begin{equation}\label{eq:posterior}
 x=\bar\mu(p),
 \qquad
 p_i=\frac{w_i e^{-\|x-\mu_i\|^2/2}}
 {\sum_jw_j e^{-\|x-\mu_j\|^2/2}}.
\end{equation}
The identity above links maxima of $f$ to the joint energy, while the local
correspondence lemma in Appendix~\ref{app:reduction} shows that
$p\mapsto\bar\mu(p)$ preserves local type, including at degenerate stationary
points. This is the reason for using $F$: it replaces the mode problem by a
constrained minimum problem on the two-dimensional simplex.
Equation~\eqref{eq:posterior} is
also the Gaussian mean-shift fixed-point relation \citep{carreira2007} and
places every stationary point on the Gaussian ridgeline
\citep{raylindsay2005}.

To make $F$ separable, define
\[
 a_i=\frac{d_{ij}^2+d_{ik}^2-d_{jk}^2}{2},
 \qquad \{i,j,k\}=\{1,2,3\}.
\]
Equivalently,
$a_i=(\mu_j-\mu_i)^{\mathsf T}(\mu_k-\mu_i)$, so its sign records whether
the angle at vertex $i$ is acute, right, or obtuse. Since
$d_{ij}^2=a_i+a_j$, \eqref{eq:profiled-potential} becomes
\begin{equation}\label{eq:potential}
 F(p)=\sum_{i=1}^3\left[
 p_i\log\frac{p_i}{w_i}+\frac{a_i}{2}p_i(1-p_i)\right].
\end{equation}
Here $\sum_i p_i\log p_i$ is negative Shannon entropy
\citep{shannon1948}, and $D_{\mathrm{KL}}$ is relative entropy
\citep{kullback1951}. The full $F$ is not an entropy, nor the differential
entropy of $f$; it is a reduced free energy with a relative-entropy term
and a quadratic interaction.

For a noncollinear configuration,
\begin{equation}\label{eq:gram}
 a_1a_2+a_2a_3+a_3a_1>0;
\end{equation}
indeed, the left-hand side is the Gram determinant of
$\mu_1-\mu_3$ and $\mu_2-\mu_3$. An acute triangle has $a_i>0$ for all
$i$; a right triangle has one zero coefficient; and an obtuse triangle has
exactly one negative coefficient.

The stationary equations for $F$ are
\begin{equation}\label{eq:stationary}
 \log\frac{p_i}{w_i}-a_ip_i+\frac{a_i}{2}=\lambda,
 \qquad \sum_i p_i=1.
\end{equation}
For $a_i>0$, put
\[
 \alpha_i=a_i^{-1},\qquad
 \eta_i=\log\frac{\alpha_i}{w_i}-1+\frac{1}{2\alpha_i}.
\]
The $i$th equation in \eqref{eq:stationary} is then
\[
 \psi\!\left(\frac{p_i}{\alpha_i}\right)=\eta_i-\lambda,
 \qquad \psi(t)=t-1-\log t.
\]
For $d>0$, let $L(d)<1<U(d)$ be the two positive solutions of
$\psi(t)=d$, and set $L(0)=U(0)=1$ by continuity. Thus the two
coordinate solutions are
\begin{equation}\label{eq:branches}
 p_i^-(\lambda)=\alpha_iL(\eta_i-\lambda),\qquad
 p_i^+(\lambda)=\alpha_iU(\eta_i-\lambda).
\end{equation}
They are analytic for $\lambda<\eta_i$ and meet at $p_i=\alpha_i$ when
$\lambda=\eta_i$. We call $\eta_i$ the barrier of coordinate $i$, and call
the two solutions its lower and upper branches. If $a_i\leq0$, the coordinate
equation instead has one positive solution, increasing and analytic on the
whole real line; we also denote this unique lower solution by
$p_i^-(\lambda)$.

Away from a branch junction, differentiating the coordinate equation gives
\begin{equation}\label{eq:hessian}
 \begin{gathered}
 p_i'=\left(\frac1{p_i}-a_i\right)^{-1},\\
 Q(v)=\sum_i\frac{v_i^2}{p_i'},\qquad \sum_i v_i=0,
 \end{gathered}
\end{equation}
where $Q$ is the constrained Hessian. Thus lower coordinates have positive
slope, whereas an upper coordinate has negative slope. At a nondegenerate
minimum, at most one coordinate can be upper.

Let $S_0(\lambda)$ be the sum of all lower coordinate solutions.  Let
$S_i(\lambda)$ be the sum with coordinate $i$ upper and all other
coordinates lower; it is defined only when $a_i>0$. A point on one of
these coordinate branches lies in $\mathcal P$ exactly when its coordinate
sum is one. Thus the simplex constraint becomes $S_i(\lambda)=1$. Since
\begin{equation}\label{eq:sumderivative}
 S_i'=\sum_jp_j',
\end{equation}
$S_0$ is increasing and contributes at most one minimum. When coordinate
$i$ is upper, the scalar minimum criterion in Appendix~\ref{app:reduction}
states that $Q$ is positive definite exactly when $S_i'<0$. Thus we must
count downward crossings of $S_i(\lambda)=1$.

We first work with noncollinear means, nonzero $a_i$, distinct barriers
and inverse coefficients for the positive coordinates, and stationary
points with nonsingular constrained Hessian that do not lie at branch
junctions. These conditions hold on a dense set of parameters.
Appendix~\ref{app:genericity} combines the known
finiteness theorem with persistence of modes and removes every genericity
assumption.

\smallskip
\noindent\textbf{Counting notation and target.}
Relabel the positive coordinates so that $\eta_1<\eta_2$ and, in the acute
case, $\eta_2<\eta_3$. Let $m_0$ be the number of solutions of
$S_0(\lambda)=1$, and, for each defined upper branch $S_i$, let $m_i$ be the
number of solutions of $S_i(\lambda)=1$ with $S_i'(\lambda)<0$, that is, the
number of downward crossings. In the generic setup these counts are precisely
the local minima of $F$ represented by the corresponding branches. Thus the
total is $m_0+m_1+m_2$ in the obtuse case and
$m_0+m_1+m_2+m_3$ in the acute case.

Sections~\ref{sec:curvature} and~\ref{sec:folds} will give
$m_0+m_1\leq2$ and $m_2\leq2$; the critical-height comparison then yields the
essential joint improvement $m_0+m_1+m_2\leq3$. In the acute case, the
remaining branch satisfies $m_3\leq1$, giving four. The limiting argument in
Appendix~\ref{app:genericity} then removes genericity and covers right and
collinear configurations.

\section{A shape lemma for two lower coordinates}
\label{sec:curvature}

Fix an upper branch $S_i$, and let
$G(\lambda)=\sum_{j\ne i}p_j^-(\lambda)$ be the sum of its two lower
coordinates. Their slopes are positive and increasing, so $G',G''>0$.
The single auxiliary function needed below is
\begin{equation}\label{eq:Kdef}
 K(\lambda)=\frac{G'(\lambda)}
 {1+\sqrt{G''(\lambda)/G'(\lambda)}}.
\end{equation}
The next section constructs a fold curve that increases exactly when
$K>\alpha_i$. Thus the following interval property is the only curvature
information needed in the main argument.

\begin{lemma}\label{lem:curvature}
For every upper branch $S_i$ in the noncollinear generic setup,
$\{\lambda:K(\lambda)>\alpha_i\}$ is an interval.
\end{lemma}

\begin{proof}
The key point is that $K$ has no local minimum at a level relevant to the
superlevel set.  We give the calculation because it is the main shape input
to the branch count.

For either lower coordinate, write $a$ for its coefficient and $p$ for its
solution, and set
\[
 z=\frac1{1-ap},\qquad \rho=p'.
\]
Differentiating the coordinate equation in \eqref{eq:stationary} gives the
common differential system
\begin{equation}\label{eq:rho-z-system}
 \rho=pz,\qquad z'=z^2(z-1),\qquad \rho'=\rho z^2.
\end{equation}
Relabel the two values of $z$ as $u>v>0$, with corresponding positive
slopes $\rho_+,\rho_-$.  Put
\[
 J=\rho_++\rho_-=G',\qquad
 r=\sqrt{J'/J}=\sqrt{G''/G'},
\]
so that $K=J/(1+r)$.  At a fixed $\lambda$, let $X$ take the values $u,v$
with probabilities $\rho_+/J,\rho_-/J$.  Then
$\mathbb E[X^2]=r^2$ and $v<r<u$.

Set $s=u+v$, $t=uv$, $N=(r+u)(r+v)$, and
\begin{equation}\label{eq:delta-Z}
 \delta=u+v-r,\qquad
 Z=\frac1u+\frac1v-\frac1r-\frac3{3\delta-2}.
\end{equation}
The two-point identities
\[
 \operatorname{Var}(X^2)=N(r\delta-t),
 \qquad
 \mathbb E[X^3]+r^3=\frac{N(sr-t)}s
\]
follow without a fourth-moment expansion.  Indeed,
$(X^2-u^2)(X^2-v^2)=0$ gives
\[
 \mathbb E[X^4]=(u^2+v^2)r^2-u^2v^2,
\]
and subtracting $r^4$ yields the variance identity.  Likewise,
$X^2=sX-t$ gives
\[
 \mathbb E[X]=\frac{r^2+t}{s},\qquad
 \mathbb E[X^3]=sr^2-\frac{t(r^2+t)}s,
\]
which factors into the displayed third-moment identity.  The probabilities
defining $X$ also depend on $\lambda$, so we differentiate $J'$ explicitly:
\begin{align*}
 J'&=\rho_+u^2+\rho_-v^2,\\
 J''&=\rho_+(3u^4-2u^3)\\
    &\quad+\rho_-(3v^4-2v^3).
\end{align*}
Hence
\[
 2rr'=\frac{J''}{J}-\left(\frac{J'}J\right)^2
 =3\mathbb E[X^4]-2\mathbb E[X^3]-r^4,
\]
and therefore
\begin{align}
 2r\{r'-r^2(1+r)\}
 &=3\operatorname{Var}(X^2)-2\{\mathbb E[X^3]+r^3\}\notag\\
 &=\frac{rtN(3\delta-2)}s Z.
 \label{eq:rprime-Z-main}
\end{align}
Consequently, on the region $\delta>1$,
\begin{equation}\label{eq:Kprime-Z-main}
 \frac{K'}K
 =-\frac{uv(r+u)(r+v)(3\delta-2)}
 {2(u+v)(1+r)}Z.
\end{equation}
The prefactor is strictly positive.  Thus at a stationary point of $K$
with $\delta>1$, $Z=0$ and direct differentiation of $K=J/(1+r)$ gives
$r'=r^2(1+r)$.  At that same point, \eqref{eq:rho-z-system} gives
\begin{align*}
 \left(\frac1u+\frac1v-\frac1r\right)'
 &=-(u-1)-(v-1)+(r+1)\\
 &=3-\delta.
\end{align*}
To obtain the remaining derivative, recall $s=u+v$, $t=uv$, and
$\delta=s-r$.  The scalar differential equation gives
\[
 u'+v'=s^2(s-1)-t(3s-2).
\]
Using $s^3-r^3-\delta^3=3sr\delta$ and
$s^2+r^2-\delta^2=2sr$, we find
\begin{align*}
 \delta'-\delta^2(\delta-1)
 &=s^2(s-1)-r^2(r+1)\\
 &\quad-t(3s-2)-\delta^2(\delta-1)\\
 &=3sr\delta-2sr-t(3s-2)\\
 &=rt(3\delta-2)Z=0.
\end{align*}
Hence $\delta'=\delta^2(\delta-1)$ at the stationary point.
Differentiating \eqref{eq:delta-Z} before imposing these pointwise
relations therefore gives
\begin{align}
 Z'&=3-\delta+\frac{9\delta^2(\delta-1)}{(3\delta-2)^2}\notag\\
 &=2+(\delta-1)
 \left[\left(\frac{3\delta}{3\delta-2}\right)^2-1\right]\notag\\
 &>2.
 \label{eq:Zprime-main}
\end{align}
Differentiating \eqref{eq:Kprime-Z-main} now shows that $K''<0$ at every
stationary point for which $\delta>1$.

It remains to check that this region contains every stationary point that
could disconnect the required superlevel set.  If both lower coefficients
are positive, then $u>v>1$ and $r<u$, hence $\delta>v>1$ everywhere.  If
$a_j>0>a_k$, put $\alpha_j=a_j^{-1}$ and
$\gamma=(-a_k)^{-1}$.  In this case $u>1>v>0$ and
$\rho_+=\alpha_j(u-1)$, $\rho_-=\gamma(1-v)$.  Eliminating their ratio from
$r^2=(\rho_+u^2+\rho_-v^2)/J$ first gives
\[
 \frac{\gamma}{\alpha_j}
 =\frac{(u-r)(u+r)(u-1)}
 {(r-v)(r+v)(1-v)}.
\]
On the other hand,
\[
 \gamma-\alpha_j-K
 =\frac{\gamma(r+v)-\alpha_j(r+u)}{1+r}.
\]
Substituting the ratio and factoring gives
\begin{equation}\label{eq:mixed-cap-main}
 \gamma-\alpha_j-K
 =\frac{\alpha_j(r+u)(u-v)}
 {(1+r)(r-v)(1-v)}(\delta-1).
\end{equation}
Thus $K<\gamma-\alpha_j$ is equivalent to $\delta>1$.  By
\eqref{eq:gram},
\[
 0<\alpha_i<\gamma-\alpha_j.
\]

Finally, if a strict superlevel set of an analytic $K$ were disconnected,
the gap between two components would contain a stationary local minimum.
With two positive lower coefficients, every stationary point instead has
$K''<0$.  In the mixed case, an intervening minimum for
$\{K>\alpha_i\}$ has value at most $\alpha_i<\gamma-\alpha_j$; by
\eqref{eq:mixed-cap-main} it again lies in $\delta>1$ and has $K''<0$, a
contradiction.  Hence $\{K>\alpha_i\}$ is an interval.  Appendix~\ref{app:algebra}
records the factorization details used above.
\end{proof}

\section{Fold curves and the number of crossings}
\label{sec:folds}

This section first combines the interval property of $\{K>\alpha_i\}$ with
the branch equations to bound the stationary points of an arbitrary upper
branch. It then converts that fold bound into bounds on the branch counts
$m_i$. Only one simultaneous extremal case will remain for
Section~\ref{sec:comparison}.

Continue with a fixed upper branch $S_i$ and its lower-coordinate sum $G$.
Rather than analyze $S_i'$ directly, we solve $S_i'(\lambda)=0$ for its
already defined barrier $\eta_i$. Writing $U=U(\eta_i-\lambda)$ and using
the inverse-branch derivative gives
\[
 S_i'(\lambda)=G'(\lambda)
 -\alpha_i\frac{U}{U-1}.
\]
The derivative is negative whenever $G'\leq\alpha_i$.  If
$G'>\alpha_i$, then
\begin{align*}
 S_i'=0
 &\ \Longleftrightarrow\ G'(U-1)=\alpha_iU\\
 &\ \Longleftrightarrow\ U=\frac{G'}{G'-\alpha_i}>1\\
 &\ \Longleftrightarrow\
 \eta_i-\lambda=\psi\!\left(\frac{G'}{G'-\alpha_i}\right).
\end{align*}
Thus a zero is equivalent to intersecting the following fold curve with
the barrier level:
\begin{equation}\label{eq:foldcurve}
 \eta_i=B(\lambda),\qquad
 B(\lambda)=\lambda+\psi\!\left(\frac{G'}{G'-\alpha_i}\right).
\end{equation}
In particular, a solution of \eqref{eq:foldcurve} automatically satisfies
$\lambda<\eta_i$. Horizontal intersections of $B$ with the fixed level
$\eta_i$ locate stationary points of $S_i$. Modes are instead downward
crossings of $S_i(\lambda)=1$; bounding its stationary points bounds the
decreasing pieces on which those crossings can occur.  Put
$U_B=G'/(G'-\alpha_i)$.  Since
\[
 U_B'=-\frac{\alpha_iG''}{(G'-\alpha_i)^2},
 \qquad
 \psi'(U_B)=\frac{\alpha_i}{G'},
\]
differentiation gives
\begin{equation}\label{eq:Bprime}
 B'=1-\frac{\alpha_i^2G''}{G'(G'-\alpha_i)^2},\qquad
 \operatorname{sign}B'=\operatorname{sign}(K-\alpha_i).
\end{equation}
Indeed, after taking positive square roots, $B'>0$ is equivalent to
\[
 G'-\alpha_i>\alpha_i\sqrt{G''/G'},
\]
which is exactly $K>\alpha_i$.  The same equivalence holds with equality
or with the inequality reversed.

By \eqref{eq:rho-z-system}, both lower slopes are strictly increasing, so
$G'$ is strictly increasing.  Moreover, $G'\to0$ as
$\lambda\to-\infty$ and $G'\to\infty$ at the first positive-coefficient
lower barrier.  Thus the fold domain $G'>\alpha_i$ is an interval. As
$G'\downarrow\alpha_i$, one has $U_B\to\infty$ and
$B\to\infty$. If lower coordinate $j$ has the smallest barrier among the
positive-coefficient lower coordinates, its $z$ tends to infinity while
the other lower coordinate remains regular. Consequently,
\[
 G'\sim\alpha_jz,\qquad G''\sim\alpha_jz^3,
 \qquad U_B\to1,
\]
and at the other endpoint
\begin{equation}\label{eq:Bendpoints}
 B(\eta_j-)=\eta_j,\qquad K(\eta_j-)=\alpha_j.
\end{equation}
We will also need the height of $S_i$ at a stationary point. On the fold,
its derivative with respect to the fold parameter is
\begin{equation}\label{eq:criticalheight}
 \frac{d}{d\lambda}\left[
 G(\lambda)+\alpha_i\frac{G'(\lambda)}{G'(\lambda)-\alpha_i}
 \right]=G'(\lambda)B'(\lambda).
\end{equation}
Thus the fold curve and its stationary height increase or decrease together.
Appendix~\ref{app:branches} records the remaining endpoint details.

By Lemma~\ref{lem:curvature}, $\{K>\alpha_i\}$ is an interval. Near the
left endpoint, where $G'\downarrow\alpha_i$, one has $K<\alpha_i$.
Hence $B$ has at most three successive monotone pieces,
with directions decrease, increase, decrease; some pieces may be absent.
Analyticity rules out a constant piece, and isolated zeros of $B'$ do not
affect strict monotonicity.  Every horizontal level therefore meets $B$
at most three times.  We have proved the following common bound.

\begin{lemma}\label{lem:threefolds}
Every upper-branch derivative has at most three distinct zeros on its
physical domain. If its fold curve ends at a lower-coordinate barrier zero
and that coordinate has inverse coefficient larger than $\alpha_i$, then
every nonnegative horizontal level meets the fold curve exactly once.
\end{lemma}

For the second assertion, $K(0-)=\alpha_j>\alpha_i$ makes
$\{K>\alpha_i\}$ a terminal interval. Thus $B$ decreases and then increases to zero; its
last increasing piece lies below zero.  This proves the assertion.

We now turn the fold bound into branch counts. Retaining the ordering above,
shift all barriers so that $\eta_1=0<\eta_2$; in the acute case, also
$\eta_2<\eta_3$. Then
\begin{align}
 S_1(\lambda)&=\alpha_1U(-\lambda)
   +\alpha_2L(\eta_2-\lambda)+p_3^-(\lambda),
 \label{eq:S1}\\
 S_2(\lambda)&=\alpha_1L(-\lambda)
   +\alpha_2U(\eta_2-\lambda)+p_3^-(\lambda).
 \label{eq:S2}
\end{align}
All physical solutions have $\lambda<0$, and $p_3^-$ is analytic near zero.
A single upper branch can contribute two minima:
Appendix~\ref{app:two-on-one} gives an exact example. Thus counting branch
types alone would not prove the theorem.

The derivative of $S_1$ is negative near both ends:
\[
 \lim_{\lambda\to-\infty}S_1'(\lambda)=-\alpha_1,
 \qquad
 \lim_{\lambda\uparrow0}S_1'(\lambda)=-\infty.
\]
Lemma~\ref{lem:threefolds} allows at most two sign changes and hence at
most two decreasing pieces. Here $S_0(0-)=S_1(0-)$. Equality with one
would place a stationary point at a branch junction and is excluded by
genericity. If this common endpoint value is below one, the all-lower branch
gives no mode. If it is above one, that branch gives one mode, but the last
decreasing piece of $S_1$ ends above one and gives no downward crossing.
Therefore
\begin{equation}\label{eq:firstgroup}
 m_0+m_1\leq2.
\end{equation}
The derivative of $S_2$ is negative near minus infinity and positive
near zero.  At most three sign changes give at most two decreasing pieces,
so
\begin{equation}\label{eq:secondgroup}
 m_2\leq2.
\end{equation}
If $\alpha_2<\alpha_1$, the second assertion of
Lemma~\ref{lem:threefolds} improves this to $m_2\leq1$.

\begin{corollary}[The last acute branch]\label{cor:lastbranch}
In the acute case, the branch $S_3$ with the largest barrier contributes
at most one minimum: $m_3\leq1$.
\end{corollary}

\begin{proof}
At every stationary point of $S_3$, the conjugate-branch inequality
$L(d)+U(d)>2$ for $d>0$ gives $S_3''>0$.
Thus every stationary point is a strict local minimum. Two such points
would force an intervening stationary local maximum, a contradiction.
Hence $S_3$ has at most one decreasing interval and at most one downward
crossing. Appendix~\ref{app:branches} gives the derivative calculation.
\end{proof}

\begin{center}
\noindent\textbf{Branch-count summary.}\par\smallskip
\parbox{\columnwidth}{\footnotesize The table records minima of $F$, not
local extrema of the scalar sums themselves. Here ``shared endpoint'' means
$S_0(0-)=S_1(0-)$, and ``opposite end signs'' means that $S_2'$ is negative
near $-\infty$ and positive near zero.}\par\smallskip
\footnotesize
{\setlength{\tabcolsep}{4pt}
\begin{tabular}{@{}lll@{}}
\toprule
Branch(es) & Bound & Reason \\
\midrule
$S_0$ & $m_0\leq1$ & monotonicity \\
$S_0,S_1$ & $m_0+m_1\leq2$ &
 \begin{tabular}[t]{@{}l@{}}fold bound;\\shared endpoint\end{tabular} \\
$S_2$ & $m_2\leq2$ &
 \begin{tabular}[t]{@{}l@{}}fold bound;\\opposite end signs\end{tabular} \\
$S_3$ (acute only) & $m_3\leq1$ & Cor.~\ref{cor:lastbranch}; $S_3''>0$ \\
\bottomrule
\end{tabular}
}
\end{center}

\noindent\textbf{The remaining gap.}
The first three bounds still allow $m_0+m_1+m_2=4$. The sole purpose of
the next section is to rule out this simultaneous extremal case when
$\alpha_2>\alpha_1$.

\section{Comparing the critical heights}
\label{sec:comparison}

\begin{figure*}[t]
 \centering
 \includegraphics[width=\textwidth]{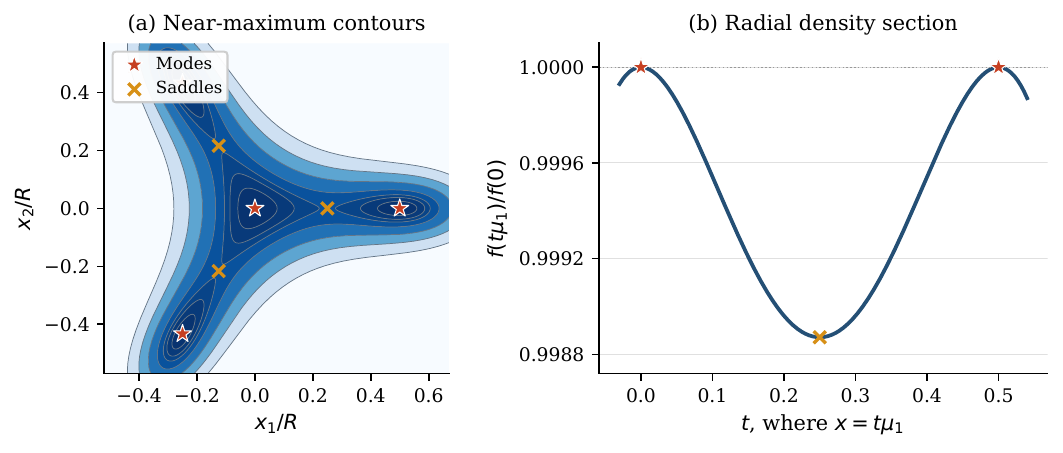}
 \caption{The sharp four-mode configuration with equal weights and
 squared circumradius $R^2=\tfrac83\log2$. The contour plot and a radial
 section illustrate the four strict modes at $0$ and $\mu_i/2$ (stars).
 Crosses indicate the saddle points $\mu_i/4$. Appendix~\ref{app:example}
 gives exact calculations; the figure is illustrative.}
 \label{fig:four}
 \figalt{Three outer modes and one central mode form a triangular arrangement. A radial section shows two equal maxima separated by a shallow saddle.}
\end{figure*}

The table leaves only one hard case: $m_0+m_1=2$ and $m_2=2$ with
$\alpha_2>\alpha_1$. The endpoint argument behind
\eqref{eq:firstgroup} shows that the first equality forces a local maximum
of $S_1$ above one; the second equality forces the last local minimum of
$S_2$ to lie below one. We compare both extrema with one equal-barrier
reference value.

Keep $\alpha_1,\alpha_2,p_3^-$ fixed and move the second barrier from
$\eta_2$ to zero. The resulting equal-barrier branch and its endpoint are
\begin{align}
 S_1^{\mathrm{eq}}(\lambda)
 &=\alpha_1U(-\lambda)+\alpha_2L(-\lambda)+p_3^-(\lambda),\notag\\
 V_0&=S_1^{\mathrm{eq}}(0-)
 =\alpha_1+\alpha_2+p_3^-(0).
 \label{eq:equalbranch}
\end{align}
The reference value $V_0$ need not equal one. We use it only as a comparison
threshold. Among the lower coordinates of $S_1^{\mathrm{eq}}$, the smallest
barrier belonging to a positive coefficient is still unique, so the fold
argument applies.

\begin{lemma}\label{lem:separation}
In the ordered setup \eqref{eq:S1}--\eqref{eq:S2}, suppose
$\alpha_2>\alpha_1$. Every stationary point of $S_1$ has
height below $V_0$. If $S_2$ has two decreasing pieces separated by an increasing
piece, its last local minimum has height above $V_0$.
\end{lemma}

\begin{proof}
For $S_1^{\mathrm{eq}}$, the first lower barrier and the upper barrier both
equal zero, and their inverse coefficients satisfy $\alpha_2>\alpha_1$.
Lemma~\ref{lem:threefolds}, applied to horizontal level zero, gives
exactly one stationary point. Since $(S_1^{\mathrm{eq}})'$ is initially
negative and finally positive, $S_1^{\mathrm{eq}}$ has one minimum and then
increases to $V_0$.
Consequently
 $(S_1^{\mathrm{eq}})'(\lambda)>0$ implies
$S_1^{\mathrm{eq}}(\lambda)<V_0$.
Both $L(d)$ and $L(d)/(1-L(d))$ decrease with $d$, so $\eta_2>0$ gives
\[
 S_1(\lambda)<S_1^{\mathrm{eq}}(\lambda),\qquad
 S_1'(\lambda)<(S_1^{\mathrm{eq}})'(\lambda).
\]
At any stationary point of $S_1$, the second inequality gives
$(S_1^{\mathrm{eq}})'>0$; the first inequality then gives
$S_1(\lambda)<V_0$.

For the second claim, two separated decreasing pieces force the derivative
signs $-,+,-,+$. Thus the last local minimum lies on the final decreasing
piece of the associated fold curve. Along that curve the stationary height
satisfies \eqref{eq:criticalheight}, so it decreases toward its
endpoint value $V_0$. Its value at the last minimum is therefore strictly
larger than $V_0$. Appendix~\ref{app:branches} gives the fold intersections
and endpoint limits in detail.
\end{proof}

If $m_0+m_1=2$, the required local maximum of $S_1$ lies above one;
Lemma~\ref{lem:separation} places it below $V_0$, so $V_0>1$.
If $m_2=2$, the last local minimum of $S_2$ lies below one; the same lemma
places it above $V_0$, so $V_0<1$. This contradiction rules out simultaneous
equality. Together with the case
$\alpha_2<\alpha_1$, this proves
\begin{equation}\label{eq:jointthree}
 m_0+m_1+m_2\leq3.
\end{equation}

For an obtuse triangle these are all branches which can contain a
minimum, proving the bound three. In the acute case,
Corollary~\ref{cor:lastbranch} adds $m_3\leq1$ to
\eqref{eq:jointthree}, proving the bound four.

Appendix~\ref{app:genericity} extends both counts to degenerate
parameters by persistence of isolated strict modes.  Every right triangle
and every collinear three-point configuration is a limit of obtuse
triangles.  This gives the remaining upper bounds in
Theorem~\ref{thm:main}.

\section{Sharpness and dimension reduction}
\label{sec:sharpness}

We finish by showing that both upper bounds are attained and by lifting
the planar result to every ambient dimension.

The classical equilateral configuration attains four. In whitened
coordinates take common covariance $I_2$, equal weights, and squared
circumradius $R^2=\tfrac83\log2$. The center and the three
points $\mu_i/2$ are stationary. Using
the posterior vector $p$ from \eqref{eq:posterior} and
$\nabla^2\log f=\operatorname{Cov}_p(\mu)-I_2$, their Hessian eigenvalues
are $\tfrac43\log2-1$ at the center and
$\tfrac43\log2-1,\tfrac23\log2-1$ at each outer point. All are negative.
Thus the four points are strict modes; their equal density values are
$2^{-4/3}/(2\pi)$. Figure~\ref{fig:four} illustrates the configuration;
Appendix~\ref{app:example} verifies the stationary points, eigenvalues,
heights, and intervening saddles exactly.

For the three-mode lower bounds, again take common covariance $I_2$, choose
three distinct means of the desired geometric class, and scale their
pairwise distances to be at least ten. Take equal weights. On the unit
circle about any one mean, the
density is at most
\[
 \frac{e^{-1/2}+2e^{-81/2}}{6\pi}
 <\frac1{6\pi},
\]
whereas at that mean it is at least $1/(6\pi)$.  The three disjoint unit
balls therefore contain three local maxima.  This works for obtuse,
right, and collinear configurations, proving sharpness in each class.

\begin{proof}[Proof of Corollary~\ref{cor:dimension}]
For the original parameters, set
\[
 y=\Sigma^{-1/2}(x-\mu_3),\qquad
 \nu_i=\Sigma^{-1/2}(\mu_i-\mu_3),
\]
and let
\[
 \widetilde f(y)=\sqrt{\det\Sigma}\,
 f(\mu_3+\Sigma^{1/2}y).
\]
Choose a two-dimensional linear subspace $P$ containing
$\operatorname{span}\{\nu_1,\nu_2\}$. For $y=u+v\in P\oplus P^\perp$,
the transformed density factors as
\[
 \widetilde f(u+v)=\varphi_{I_{P^\perp}}(v)
   \sum_{i=1}^3w_i\varphi_{I_P}(u-\nu_i).
\]
Here $\varphi_{I_P}$ denotes the standard Gaussian density on $P$. The
first factor has its unique mode at $v=0$, so modes of
$\widetilde f$ correspond exactly to modes of the mixture on $P$. The
planar theorem therefore gives the stated bounds, including the geometric
three-mode cases. The planar sharpness examples embed in every $d\geq2$.
\end{proof}

\section{Conclusion and Discussion}
\label{sec:conclusion}

We have established the sharp maximum of four modes for mixtures of three
Gaussian components sharing a positive-definite covariance matrix, in every
dimension $d\geq 2$. The result holds for arbitrary means and positive
weights, including degenerate modes. When the whitened means are collinear
or form a right or obtuse triangle, the sharp maximum is three.

The proof reduces mode counting to scalar branch analysis of a free-energy
functional on the probability simplex, using curvature control and
comparisons of stationary heights.

The broader problem of determining the sharp maximum number of modes for
general multivariate Gaussian mixtures remains open
\citep{amendola2020}. Natural directions for future work are to obtain sharp
bounds for mixtures with more than three components sharing a common
covariance, and for mixtures whose component covariance matrices differ.
Within the present model, characterizing the parameter region in which
four modes occur is another question left open by this work.
  \label{end:main}

\clearpage
\raggedbottom
\subsection*{AI Use Statement}
Generative AI (ChatGPT) was used for mathematical exploration, proposed proofs and simplifications, internal critiques, literature searches, code generation, symbolic-check design, and manuscript drafting. The algebraic verification scripts were executed and the manuscript received multiple AI-assisted internal checks. These checks are not a formal proof certification. No data involving human participants and no learned models were used. The submitting authors have reviewed the content and take responsibility for the manuscript in its entirety.
 
\bibliographystyle{plainnat}

\clearpage
\flushbottom
\appendix
\onecolumn
\section*{Supplementary proofs and exact calculations}
\section{The variational reduction and the minimum criterion}
\label{app:reduction}
The variational argument applies in any ambient dimension $d$; the main
theorem has $d=2$. To make the normalization explicit, start from
\[
 f_{\mathrm{orig}}(x)=\sum_iw_i
 \varphi_\Sigma(x-\mu_i^{\mathrm{orig}})
\]
and set
\[
 y=\Sigma^{-1/2}x,\qquad
 \nu_i=\Sigma^{-1/2}\mu_i^{\mathrm{orig}},\qquad
 \widetilde f(y)=\sqrt{\det\Sigma}\,
 f_{\mathrm{orig}}(\Sigma^{1/2}y).
\]
Then $\widetilde f(y)=\sum_iw_i\varphi_{I_d}(y-\nu_i)$; the weights are
unchanged and modes correspond under $x=\Sigma^{1/2}y$. Throughout this
appendix we relabel $(\widetilde f,y,\nu_i)$ as $(f,x,\mu_i)$. Thus the
common covariance is $I_d$.

This section justifies the reduction at degenerate critical points as well
as at nondegenerate ones. Write
\[
  \mathcal P=\{p\in(0,1)^3:p_1+p_2+p_3=1\},
  \qquad \bar\mu(p)=\sum_{i=1}^3p_i\mu_i.
\]
With $d_{ij}^2=\|\mu_i-\mu_j\|^2$ and
$a_i=(d_{ij}^2+d_{ik}^2-d_{jk}^2)/2$, the identities
$d_{ij}^2=a_i+a_j$ and
\begin{align}
 &\sum_i p_i\|\mu_i\|^2-\|\bar\mu(p)\|^2 \notag\\
 &\hspace{12mm}=\sum_{i<j}p_ip_jd_{ij}^2
   =\sum_i a_ip_i(1-p_i)
 \label{eq:appendix-variance}
\end{align}
hold even if some means coincide. Consequently, the potential can be written
in either form
\begin{align}
 F(p)
 &=\sum_i p_i\log\frac{p_i}{w_i}
   +\frac12\sum_i p_i\|\mu_i\|^2-\frac12\|\bar\mu(p)\|^2
   \notag\\
 &=\sum_i\left[
 p_i\log\frac{p_i}{w_i}+\frac{a_i}{2}p_i(1-p_i)\right].
 \label{eq:appendix-potential}
\end{align}
The first line identifies $F$ as $D_{\mathrm{KL}}(p\Vert w)$ plus half
the total $p$-variance of the whitened component means. Thus $F$ is a
reduced free-energy functional, not the entropy of the mixture density.

\begin{lemma}[Correspondence of local extrema]
\label{lem:appendix-extrema}
The stationary points of $F$ on $\mathcal P$ correspond bijectively to
the stationary points of $f$. The correspondence is
$p\mapsto x=\bar\mu(p)$, with inverse given by the posterior probabilities.
Moreover, $p$ is a local minimum of $F$ if and only if $x$ is a local
maximum of $f$. Neither nondegeneracy nor affine independence of the
means is required.
\end{lemma}

\begin{proof}
Put $\widetilde w_i=w_i\exp(-\|\mu_i\|^2/2)$ and define
\begin{align*}
 \mathcal E(x,p)
 &=\sum_i p_i\log\frac{p_i}{\widetilde w_i}
   +\frac12\|x\|^2-x\cdot\bar\mu(p)\\
 &=F(p)+\frac12\|x-\bar\mu(p)\|^2.
\end{align*}
Hence $F(p)=\min_{x\in\mathbb R^d}\mathcal E(x,p)$: the location variable has
been profiled out.
For fixed $x$, its unique minimizer over $\mathcal P$ is
\[
 p_i(x)=
 \frac{\widetilde w_i\exp(x\cdot\mu_i)}
 {\sum_j\widetilde w_j\exp(x\cdot\mu_j)}.
\]
Indeed, subtracting the minimum value leaves the relative entropy
$\sum_i p_i\log(p_i/p_i(x))$. Therefore
\begin{equation}
 \phi(x):=-\log((2\pi)^{d/2}f(x))
 =\min_{p\in\mathcal P}\mathcal E(x,p).
 \label{eq:appendix-variational}
\end{equation}
This is the finite-state Gibbs variational identity. Its maximization form
is the free-energy formulation of EM \citep{nealhinton1998}; for Gaussian
mean shift it equals $-\mathcal E$, up to the Gaussian normalization constant
\citep{carreira2006}. The posterior and mean updates give the associated
Gaussian mean-shift EM iteration \citep{carreira2007}.
Differentiating in $x$ gives
$\nabla\phi(x)=x-\bar\mu(p(x))$.
On the other hand, constrained stationarity of $F$ gives
\[
 \log(p_i/\widetilde w_i)-\mu_i\cdot\bar\mu(p)=\text{constant}.
\]
Thus a stationary $p$ satisfies $p=p(\bar\mu(p))$, and a stationary $x$
satisfies $x=\bar\mu(p(x))$. These two relations establish the stated bijection.

Suppose first that $p_*$ is a local minimum of $F$, and let
$x_*=\bar\mu(p_*)$. The posterior is continuous and $p(x_*)=p_*$, so for $x$
sufficiently close to $x_*$,
\begin{align*}
 \phi(x)
 &=F(p(x))+\frac12\|x-\bar\mu(p(x))\|^2\\
 &\ge F(p_*)=\phi(x_*).
\end{align*}
Hence $x_*$ is a local maximum of $f$. Conversely, if $x_*$ is a local
minimum of $\phi$ and $p_*=p(x_*)$, then $\bar\mu(p_*)=x_*$. For all $p$
sufficiently close to $p_*$, continuity of $\bar\mu$ and
\eqref{eq:appendix-variational} yield
\[
 F(p)=\mathcal E(\bar\mu(p),p)
 \ge\phi(\bar\mu(p))\ge\phi(x_*)=F(p_*).
\]
This proves the converse.
\end{proof}

For the whitened density and means just defined, direct differentiation
also yields
\[
 \nabla^2\log f(x)=\operatorname{Cov}_{p(x)}(\mu)-I,
\]
where
$\operatorname{Cov}_{p}(\mu)
=\sum_i p_i(\mu_i-\bar\mu(p))(\mu_i-\bar\mu(p))^{\mathsf T}$.
The variational argument above is stronger than a comparison of these
Hessians, since it also covers a singular Hessian.

\begin{lemma}[The scalar minimum criterion]
\label{lem:appendix-minimum}
Consider a stationary point for which all
$h_i=1/p_i-a_i$ are nonzero. If all $h_i>0$, it is a strict local
minimum of $F$. If at least two $h_i<0$, it is not a local minimum.
If exactly one $h_i<0$, then its constrained Hessian is positive
definite if and only if $S_i'(\lambda)<0$, where $S_i$ is the branch
sum with coordinate $i$ upper and the other coordinates lower.
\end{lemma}

\begin{proof}
The constrained Hessian is
\[
 Q(v)=\sum_i h_iv_i^2,\qquad \sum_i v_i=0.
\]
Positivity when every $h_i>0$ is immediate. If $h_i,h_j<0$, the tangent
vector with $v_i=1$, $v_j=-1$, and the remaining entry zero gives
$Q(v)<0$.

Suppose now that only $h_i$ is negative, and call the other indices
$j,k$. Eliminate $v_i=-v_j-v_k$. In the independent coordinates
$(v_j,v_k)$, the quadratic form is
\begin{align*}
 Q(v)&=h_i(v_j+v_k)^2+h_jv_j^2+h_kv_k^2\\
 &=(h_i+h_j)v_j^2+2h_iv_jv_k+(h_i+h_k)v_k^2.
\end{align*}
Its symmetric matrix and determinant are
\begin{align}
 M&=\begin{pmatrix}h_i+h_j&h_i\\h_i&h_i+h_k\end{pmatrix},\notag\\
 \det M&=(h_i+h_j)(h_i+h_k)-h_i^2\notag\\
 &=h_ih_j+h_ih_k+h_jh_k\notag\\
 &=h_ih_jh_k\left(\frac1{h_i}+\frac1{h_j}+\frac1{h_k}\right)
 =h_ih_jh_k S_i'.
 \label{eq:appendix-hessian-determinant}
\end{align}
In direction $(1,-1)$ the value is $h_j+h_k>0$, so $M$ has at least
one positive eigenvalue. Its two eigenvalues are both positive exactly
when their product, $\det M$, is positive. Because $h_ih_jh_k<0$,
this is equivalent to $S_i'<0$.
\end{proof}

Thus, at parameters with nondegenerate stationary points and no $h_i=0$,
the all-lower branch contributes at most one minimum, and every other
minimum is a downward crossing of an upper-branch sum through one.

\section{Finiteness, generic parameters, and limiting configurations}
\label{app:genericity}

\begin{lemma}[Known finiteness theorem]
\label{lem:appendix-finite}
For any three means and positive weights, the stationary points of $F$
and of $f$ are finite, including tied barriers, right or collinear
configurations, and repeated means.
\end{lemma}

\begin{proof}
After whitening, \citet[Theorem~1]{okuno2026} gives at most fifteen
critical points of $f$, including degenerate points and repeated or
collinear centers. Lemma~\ref{lem:appendix-extrema} transfers this
finiteness to $F$. Only finiteness, not the numerical bound fifteen,
is used below.
\end{proof}

\begin{lemma}[Persistence of modes]
\label{lem:appendix-persistence}
Every mode is isolated and strict. Given a mixture with $k$ modes,
every sufficiently small perturbation of its means and positive weights
has at least $k$ modes.
\end{lemma}

\begin{proof}
If a local maximum were not strict, there would be equal-valued points
arbitrarily close to it inside a neighborhood on which its value is
maximal. Each such point would itself be a local maximum and hence a
stationary point. This contradicts Lemma~\ref{lem:appendix-finite}.

Choose disjoint small closed balls about the finitely many modes. On
each boundary, the density is strictly smaller than at the corresponding
center. Compactness makes each gap positive. The mixture density depends
continuously, uniformly on the union of these balls, on the means and
weights. All boundary gaps therefore persist under sufficiently small
perturbations. On each ball the perturbed density attains its maximum
in the interior, giving at least one mode in each ball.
\end{proof}

We next justify every generic restriction used in the branch count.
Fix the coefficients $a_i$ and use $p_1,p_2$ as coordinates on
$\Omega=\{p_1>0,p_2>0,p_1+p_2<1\}$, with $p_3=1-p_1-p_2$.
Set $b_i=\log w_i-a_i/2$ and let
\[
 F_a(p)=\sum_i\left(p_i\log p_i-\frac{a_i}{2}p_i^2\right),
 \qquad c=(b_1-b_3,b_2-b_3).
\]
Then $F=F_a-c\cdot(p_1,p_2)-b_3$, so stationarity is equivalent to
\[
 \nabla F_a(p_1,p_2)=c.
\]
Sard's theorem implies that almost every $c\in\mathbb R^2$ is a regular
value of $\nabla F_a$. At every stationary point for such a field, the
constrained Hessian is nonsingular. The vector $c$ varies freely over
$\mathbb R^2$ as the two log-weight ratios vary.

The additional exclusions also have measure zero in field space.
If $a_i>0$, the condition $h_i=0$ is the line $p_i=1/a_i$ intersected
with $\Omega$; otherwise it is empty. The image of each such line under
the smooth map $\nabla F_a$ has planar measure zero, as seen by exhausting
the line by countably many compact segments. Equalities between two
positive-coordinate barriers are affine lines in $c$ and can likewise
be avoided. Hence arbitrarily small weight perturbations make every
stationary point nondegenerate, make all relevant $h_i$ nonzero, and
make the positive-coordinate barriers distinct.

Finally, within the open set of acute triangles, or within the open set
of obtuse triangles, an arbitrarily small perturbation of the vertices
makes the positive coefficients $a_i$, and thus their inverses
$\alpha_i$, pairwise distinct. Indeed, equalities $a_i=a_j$ are the
corresponding equal-squared-distance equations and have empty interior.
After this geometry perturbation, choose a nearby generic field as above.
Lemma~\ref{lem:appendix-persistence} shows that any mode bound proved for
these generic parameters holds for every acute or obtuse configuration.
It also covers stationary points lying at a branch endpoint and modes
with a degenerate Hessian.

To pass to the remaining geometries, a right triangle is a limit of
obtuse triangles. A collinear triple with three distinct means has a
middle vertex of angle $\pi$; moving this vertex a sufficiently small
distance off the line gives an obtuse triangle. Every collinear triple
with repeated means is itself a limit of triples with three distinct
collinear means, and hence a limit of obtuse triangles as well.
Keep the weights fixed in these approximations. Finiteness and strictness
hold also at the limiting configuration, so
Lemma~\ref{lem:appendix-persistence} transfers the obtuse upper bound
to every right or collinear configuration.
 \section{Scalar derivatives and the fold curve, step by step}
\label{app:branches}

This appendix gives the one-variable calculations behind the branch
derivatives, the fold curve, and the final comparison.

\subsection{Inverse branches and a common differential equation}

For $\psi(t)=t-1-\log t$,
\[
 \psi'(t)=1-\frac1t=\frac{t-1}{t}.
\]
Implicit differentiation of $\psi(t(d))=d$ gives
\begin{equation}
 \frac{dt}{dd}=\frac{t}{t-1}.
 \label{eq:app-inverse-derivative}
\end{equation}
Thus $L_d=-L/(1-L)<0$ and $U_d=U/(U-1)>0$.
If $d=\eta-\lambda$, the chain rule contributes a minus sign:
\begin{align}
 \frac{d}{d\lambda}[\alpha L(\eta-\lambda)]
  &=\frac{\alpha L}{1-L}>0,\notag\\
 \frac{d}{d\lambda}[\alpha U(\eta-\lambda)]
  &=-\frac{\alpha U}{U-1}<0.
 \label{eq:app-branch-derivatives}
\end{align}
One differential system handles every coordinate: both signs of $a$
and both inverse branches. Away from $1-ap=0$, differentiate the coordinate
equation \eqref{eq:stationary} and define
\[
 z=\frac1{1-ap},\qquad \rho=p'.
\]
Differentiating that equation and then $z$ gives
\begin{align*}
 (1/p-a)p'&=1,\\
 \rho=p'&=\frac{p}{1-ap}=pz,\\
 z'&=\frac{a p'}{(1-ap)^2}
     =apz^3=(1-z^{-1})z^3=z^2(z-1),\\
 \rho'&=(pz)'=p'z+pz'\\
   &=pz^2+pz^2(z-1)=pz^3=\rho z^2.
\end{align*}
These identities do not assume that $z$ or $\rho$ is positive. On a lower
branch, $z,\rho>0$ and hence $\rho'=\rho z^2>0$. On an upper branch,
$z,\rho<0$;
the same equations remain valid. We will reuse them for the last upper
branch below.

For $a\ne0$, the identity $ap=1-z^{-1}$ also gives
$\rho=(z-1)/a$.
In particular, for lower coordinates,
\[
 \rho=\begin{cases}
 \alpha(z-1),&a>0,\quad\alpha=1/a,\quad z>1,\\
 \gamma(1-z),&a=-c<0,\quad\gamma=1/c,\quad 0<z<1.
 \end{cases}
\]
Even $a=0$ satisfies the differential system: $z=1$, $p'=p$, and
$\rho'=\rho$. The generic branch count excludes that case only temporarily.

For an obtuse triangle $a_1,a_2>0$ and $a_3=-c<0$. The geometric
constraint becomes
\begin{align*}
 0&<a_1a_2-c(a_1+a_2),\\
 \frac1c&>\frac1{a_1}+\frac1{a_2},\\
 \gamma&>\alpha_1+\alpha_2.
\end{align*}
Consequently, if $i$ is upper and $j$ is the positive lower coordinate,
$\alpha_i<\gamma-\alpha_j$. This is the cap used in the proof of
Lemma~\ref{lem:curvature}.

\subsection{Solving for a zero of the branch derivative}

Fix an upper coordinate $i$, and let $G$ be the sum of its two lower
coordinates. Writing $U=U(\eta_i-\lambda)$,
\eqref{eq:app-branch-derivatives} gives
\[
 S_i'=G'-\alpha_i\frac{U}{U-1}.
\]
If $G'\le\alpha_i$, then $U/(U-1)>1$ implies $S_i'<0$. If
$G'>\alpha_i$,
\begin{align*}
 S_i'=0
 &\ \Longleftrightarrow\ G'(U-1)=\alpha_i U\\
 &\ \Longleftrightarrow\ U=\frac{G'}{G'-\alpha_i}>1\\
 &\ \Longleftrightarrow\ \eta_i-\lambda
      =\psi\!\left(\frac{G'}{G'-\alpha_i}\right).
\end{align*}
Therefore a derivative zero is equivalent to $\eta_i=B(\lambda)$, with
\[
 B(\lambda)=\lambda+\psi(U_B(\lambda)),
 \qquad U_B(\lambda)=\frac{G'}{G'-\alpha_i}.
\]
Since $U_B>1$ and $\psi$ has its strict minimum $\psi(1)=0$, we have
$B(\lambda)>\lambda$. Thus every solution of $\eta_i=B(\lambda)$ lies
in the physical upper-branch domain $\lambda<\eta_i$.

Differentiate the two factors separately:
\begin{align*}
 U_B'&=-\frac{\alpha_iG''}{(G'-\alpha_i)^2},\\
 \psi'(U_B)&=1-\frac1{U_B}=\frac{\alpha_i}{G'}.
\end{align*}
It follows that
\begin{align}
 B'&=1+\psi'(U_B)U_B'\notag\\
   &=1-\frac{\alpha_i^2G''}{G'(G'-\alpha_i)^2},
 \label{eq:app-fold-sign}
\end{align}
where $K=G'/(1+\sqrt{G''/G'})$. Since all displayed factors are positive,
\begin{align*}
 B'>0
 &\ \Longleftrightarrow\ G'-\alpha_i>
      \alpha_i\sqrt{G''/G'}\\
 &\ \Longleftrightarrow\ K>\alpha_i.
\end{align*}
The same equivalences hold with equality or with the inequalities
reversed. Thus $\operatorname{sign}B'=\operatorname{sign}(K-\alpha_i)$.

Since $\rho'=\rho z^2>0$ for both lower coordinates, $G'$ is strictly
increasing. It tends to zero as $\lambda\to-\infty$ and to infinity at
the first barrier among positive-coefficient lower coordinates. Thus the fold domain
$G'>\alpha_i$ is an interval. As $G'\downarrow\alpha_i$, one has
$U_B\to\infty$, $B\to\infty$, and $K<\alpha_i$.

If lower coordinate $j$ has the unique first positive barrier, its $z$
tends to infinity while the other lower coordinate remains regular.
Therefore
\begin{align*}
 G'&\sim\alpha_j z,&
 G''&\sim\alpha_j z^3,&
 \sqrt{G''/G'}&\sim z,\\
 K&\longrightarrow\alpha_j,&
 U_B&\longrightarrow1,&
 B&\longrightarrow\eta_j.
\end{align*}
These are the endpoint statements used in Section~\ref{sec:folds}.
The temporary uniqueness assumption is part of the generic argument;
Appendix~\ref{app:genericity} removes it for the theorem.

One endpoint observation includes the horizontal level zero. If
$\eta_j=0$ and $\alpha_j>\alpha_i$, the interval
$\{K>\alpha_i\}$ is terminal, because $K(0-)=\alpha_j>\alpha_i$.
Thus $B$ first decreases and
then increases to zero. Its final increasing piece lies strictly below
zero, while its initial decreasing piece starts at $+\infty$.
Consequently every nonnegative horizontal level, including zero, meets
$B$ exactly once. This is the endpoint case of
Lemma~\ref{lem:threefolds} used below.

\subsection{Critical heights along the fold curve}

The fold curve records where a branch sum is stationary. Its height at
those stationary points is controlled by the same derivative $B'$.
At $\eta_i=B(\lambda)$ we already know that
$U=U_B=G'/(G'-\alpha_i)$, so the stationary height is
$G+\alpha_iG'/(G'-\alpha_i)$. Direct differentiation yields
\begin{equation}
 \frac{d}{d\lambda}\left[G+\alpha_i\frac{G'}{G'-\alpha_i}\right]
 =G'-\frac{\alpha_i^2G''}{(G'-\alpha_i)^2}=G'B'.
 \label{eq:app-critical-height}
\end{equation}
Since $G'>0$, the stationary height increases or decreases on precisely
the same pieces as $B$. Isolated zeros of $B'$ do not affect
strict monotonicity. For example, on a decreasing piece and for $x<y$,
\[
 \int_x^y G'(\lambda)B'(\lambda)\,d\lambda<0.
\]
Here $B'\leq0$, it is not identically zero on any subinterval, and $G'>0$.
No simplicity assumption on the fold's stationary points is needed.

\subsection{The two height comparisons}

Suppose $\alpha_2>\alpha_1$, and use the equal-barrier branch
$S_1^{\mathrm{eq}}$ and endpoint $V_0$ from \eqref{eq:equalbranch}.
For this branch, the upper barrier and the first lower barrier are both zero.
Their inverse coefficients satisfy $\alpha_1<\alpha_2$.
The nonnegative-level observation above therefore gives exactly one
stationary point. Its derivative is negative near minus infinity
and positive near zero. For the latter sign,
$\psi(1+\varepsilon)=\varepsilon^2/2+O(\varepsilon^3)$
gives the branch derivative asymptotics
\[
 (S_1^{\mathrm{eq}})'(\lambda)
 =\frac{\alpha_2-\alpha_1}{\sqrt{-2\lambda}}+O(1)>0
 \quad(\lambda\uparrow0),
\]
since $(p_3^-)'$ stays bounded. Thus its stationary point is a minimum.
After it, $S_1^{\mathrm{eq}}$ increases strictly to $V_0$. In particular,
\[
 (S_1^{\mathrm{eq}})'(\lambda)>0
 \quad\Longrightarrow\quad S_1^{\mathrm{eq}}(\lambda)<V_0.
\]
For $\eta_2>0$, monotonicity of the lower branch gives
\begin{align*}
 L(\eta_2-\lambda)&<L(-\lambda),\\
 \frac{L(\eta_2-\lambda)}{1-L(\eta_2-\lambda)}
   &<\frac{L(-\lambda)}{1-L(-\lambda)}.
\end{align*}
Consequently
\[
 S_1(\lambda)<S_1^{\mathrm{eq}}(\lambda),\qquad
 S_1'(\lambda)<(S_1^{\mathrm{eq}})'(\lambda).
\]
At every stationary point of $S_1$, we have
$(S_1^{\mathrm{eq}})'>0$ and hence
$S_1(\lambda)<S_1^{\mathrm{eq}}(\lambda)<V_0$. This proves the first part of
Lemma~\ref{lem:separation}.

For the second part, suppose $S_2$ has two decreasing pieces separated by
an increasing piece, as is necessary for two downward crossings of level one. Its derivative is
negative near minus infinity and positive near zero. Since it has at
most three zeros, its successive signs must be $-,+,-,+$. Define
\[
 G_2(\lambda)=\alpha_1L(-\lambda)+p_3^-(\lambda),
\]
and, on $G_2'>\alpha_2$, define its fold curve by
\[
 B_2(\lambda)=\lambda
 +\psi\!\left(\frac{G_2'(\lambda)}{G_2'(\lambda)-\alpha_2}\right).
\]
Thus $S_2'(\lambda)=0$ exactly when
$\eta_2=B_2(\lambda)$, so the horizontal line of height $\eta_2$ meets $B_2$
three times. There is one
intersection on each of its three monotone pieces. The last local
minimum of $S_2$, denoted by $\lambda_A$, lies on the final decreasing
piece of $B_2$.

Use the critical-height formula with $G=G_2$ and upper coefficient
$\alpha_2$. On the final decreasing piece of $B_2$, the stationary height decreases
strictly by \eqref{eq:app-critical-height}. As $\lambda\uparrow0$,
its endpoint is explicit:
\begin{align*}
 G_2(\lambda)&\longrightarrow\alpha_1+p_3^-(0),\\
 G_2'(\lambda)&\longrightarrow\infty,\\
 \alpha_2\frac{G_2'(\lambda)}{G_2'(\lambda)-\alpha_2}
   &\longrightarrow\alpha_2,\\
 G_2(\lambda)+\alpha_2\frac{G_2'(\lambda)}
 {G_2'(\lambda)-\alpha_2}&\longrightarrow V_0.
\end{align*}
At $\lambda_A$, the equality $B_2(\lambda_A)=\eta_2$ gives
\[
 S_2(\lambda_A)>V_0.
\]
This proves the second height comparison directly from the fold curve.

\subsection{The last upper branch: inspect only its stationary points}

For the final acute-case branch, we show directly that every stationary
point of $S_3$ is a strict local minimum. The common differential system
above then avoids introducing and differentiating ratios of branches.

First we need one elementary inequality. If $0<\ell<1<u$ and
$\psi(\ell)=\psi(u)$, then $\ell+u>2$.
The identity $\psi(\ell)=\psi(u)$ gives
\[
 \log u-\log \ell=u-\ell.
\]
Since $x\mapsto1/x$ is strictly convex on $(0,\infty)$ and $\ell\ne u$,
strict Jensen's inequality for the uniform distribution on $[0,1]$ gives
\begin{align*}
 1
 &=\frac{\log u-\log \ell}{u-\ell}
   =\int_0^1\frac{dt}{\ell+t(u-\ell)}\\
 &>\frac{1}{\int_0^1[\ell+t(u-\ell)]\,dt}
   =\frac2{\ell+u}.
\end{align*}
Multiplication by $\ell+u>0$ proves the claim in a single argument for all
$0<\ell<1<u$.
Now set
\[
 \ell_3=L(\eta_3-\lambda),\qquad
 u_3=U(\eta_3-\lambda),\qquad
 \ell_i=L(\eta_i-\lambda)\quad(i=1,2).
\]
Since $\eta_i<\eta_3$ and $L$ is decreasing,
\[
 \ell_i>\ell_3,\qquad 1-\ell_i<1-\ell_3<u_3-1.
\]
On this branch $a_ip_i=\ell_i$ for $i=1,2$ and $a_3p_3=u_3$.
Therefore
\[
 z_i=\frac1{1-\ell_i}>0\quad(i=1,2),\qquad
 z_3=\frac1{1-u_3}<0,
\]
and $z_i>|z_3|$ for $i=1,2$.
The coordinate slopes have signs $\rho_1,\rho_2>0>\rho_3$. At a stationary
point of $S_3=p_1+p_2+p_3$, its first derivative gives
\[
 0=S_3'=\rho_1+\rho_2+\rho_3,
 \qquad \rho_3=-(\rho_1+\rho_2).
\]
Using $\rho_i'=\rho_i z_i^2$ for all three coordinates, we obtain
\begin{align*}
 S_3''
 &=\rho_1z_1^2+\rho_2z_2^2+\rho_3z_3^2\\
 &=\rho_1(z_1^2-z_3^2)+\rho_2(z_2^2-z_3^2)\\
 &>0.
\end{align*}

There cannot be two stationary points. Indeed, just after the first one
$S_3'>0$, whereas just before the second one $S_3'<0$. Choose points
$x<y$ with these respective signs. A maximum of $S_3$ on $[x,y]$ must
then occur at an interior point $c$: neither endpoint can be a maximum
because of its derivative sign. At this point $S_3'(c)=0$ and
$S_3''(c)\leq0$, contradicting the strict inequality just proved.

Thus $S_3$ has at most one stationary point, and at that point its
derivative changes from negative to positive. If it has no stationary
point, its derivative has a constant sign. In either case there is at
most one decreasing interval. On that interval $S_3=1$ has at most one
downward crossing, so $m_3\leq1$. Coincident barriers are handled by
the perturbation argument in Appendix~\ref{app:genericity}.
 \section{The curvature calculation through two-point variance}
\label{app:algebra}
\label{app:curvature}

This appendix records the full algebra behind the curvature proof in
Lemma~\ref{lem:curvature}. In particular, it derives the factorization of
$K'/K$, the stationary-point identities, and the mixed-sign cap without
suppressing intermediate steps. The main algebraic step uses only the
variance and third moment of a two-point distribution. Primes throughout
mean differentiation with respect to $\lambda$.

\subsection{Definitions and the identity to prove}

Suppose that $u>v>0$ and $\rho_+,\rho_->0$ satisfy
\begin{align}
 u'&=u^2(u-1), & v'&=v^2(v-1),\\
 \rho_+'&=\rho_+u^2, & \rho_-'&=\rho_-v^2.
\end{align}
Set
\begin{equation}
 J=\rho_++\rho_-,\qquad r=\sqrt{J'/J},\qquad K=\frac{J}{1+r}.
 \label{eq:app-definitions}
\end{equation}
When $\rho_+,\rho_-$ are the two lower-coordinate slopes and $G$ is
their coordinate sum, $J=G'$ and $J'=G''$. Thus this definition of $K$
is exactly \eqref{eq:Kdef}.
In particular,
\[
 r^2=\frac{\rho_+u^2+\rho_-v^2}{J},\qquad v<r<u.
\]
Define
\begin{align}
 \delta&=u+v-r,\notag\\
 Z&=\frac1u+\frac1v-\frac1r-\frac3{3\delta-2}.
 \label{eq:app-normalized-Z}
\end{align}
We use these quantities on the open region where $\delta>1$.
The identity we need is
\begin{align}
 \frac{K'}K
 &=-\frac{uv(r+u)(r+v)(3\delta-2)}
 {2(u+v)(1+r)}\,Z.
 \label{eq:app-Kprime}
\end{align}
We derive this identity on the whole region $\delta>1$, before imposing
$K'=0$.

\subsection{A derivative expressed as variance}

For brevity, write
\[
 s=u+v,\qquad t=uv,\qquad N=(r+u)(r+v).
\]
Here $t$ denotes a product, not a differentiation variable. At each
fixed $\lambda$, let $X$ take the values $u,v$ with probabilities
$\rho_+/J,\rho_-/J$, respectively. Then $\mathbb E[X^2]=r^2$.

These probabilities depend on $\lambda$. We account for that dependence
by differentiating $J'$ explicitly:
\begin{align}
 J'&=\rho_+u^2+\rho_-v^2,\notag\\
 J''&=\rho_+'u^2+2\rho_+u\,u'
       +\rho_-'v^2+2\rho_-v\,v'\notag\\
    &=\rho_+(3u^4-2u^3)+\rho_-(3v^4-2v^3).
 \label{eq:app-Jsecond}
\end{align}
Differentiating $r^2=J'/J$ now gives
\begin{align}
 2rr'
   &=\frac{J''}{J}-\left(\frac{J'}J\right)^2\notag\\
   &=3\mathbb E[X^4]-2\mathbb E[X^3]-r^4.
\end{align}
Subtract $2r^3(1+r)$ and use
$\operatorname{Var}(X^2)=\mathbb E[X^4]-r^4$:
\begin{equation}
 2r\{r'-r^2(1+r)\}
 =3\operatorname{Var}(X^2)-2(\mathbb E[X^3]+r^3).
 \label{eq:app-variance-derivative}
\end{equation}
We next compute the two terms on the right without expanding a fourth
moment.

\subsection{Two elementary identities}

First, $X^2$ takes just the two values $u^2,v^2$. Thus
$(X^2-u^2)(X^2-v^2)=0$. Taking its expectation gives
\[
 \mathbb E[X^4]=(u^2+v^2)r^2-u^2v^2.
\]
Subtracting $r^4$ and factoring differences of squares yields
\begin{align}
 \operatorname{Var}(X^2)
   &=(u^2-r^2)(r^2-v^2)\notag\\
   &=(r+u)(r+v)(u-r)(r-v)\notag\\
   &=N(r\delta-t).
 \label{eq:app-two-point-variance}
\end{align}
The last equality uses
$(u-r)(r-v)=sr-r^2-t=r\delta-t$.

Second, both possible values of $X$ satisfy $X^2=sX-t$, so
\begin{align}
 \mathbb E[X]&=\frac{r^2+t}{s},\notag\\
 \mathbb E[X^3]
   &=s\mathbb E[X^2]-t\mathbb E[X]\notag\\
   &=sr^2-\frac{t(r^2+t)}{s}.
\end{align}
Consequently,
\begin{align}
 \mathbb E[X^3]+r^3
   &=\frac{sr^3+(s^2-t)r^2-t^2}{s}\notag\\
   &=\frac{(r^2+sr+t)(sr-t)}{s}\notag\\
   &=\frac{N(sr-t)}{s}.
 \label{eq:app-third-moment}
\end{align}
To verify the middle line, multiply the two factors: the terms
$-srt$ and $+srt$ cancel. The last line follows from
$r^2+sr+t=(r+u)(r+v)=N$.

Substituting \eqref{eq:app-two-point-variance} and
\eqref{eq:app-third-moment} into \eqref{eq:app-variance-derivative}
now gives
\begin{align}
 2r\{r'-r^2(1+r)\}
   &=\frac Ns\{3s(r\delta-t)-2(sr-t)\}\notag\\
   &=\frac Ns\{sr(3\delta-2)-t(3s-2)\}.
 \label{eq:app-factorization}
\end{align}
Since $s/t=1/u+1/v$ and $\delta=s-r$, the expression in braces is
\begin{align}
 rt(3\delta-2)Z
   &=(rs-t)(3\delta-2)-3rt\notag\\
   &=rs(3\delta-2)-t(3s-2).
 \label{eq:app-factor-Z}
\end{align}
Hence
\begin{equation}
 r'=r^2(1+r)+\frac{tN(3\delta-2)}{2s}\,Z.
 \label{eq:app-rprime-Z}
\end{equation}
Finally, direct differentiation of $K$ proves the desired identity:
\begin{align}
 \frac{K'}K
   &=\frac{J'}J-\frac{r'}{1+r}\notag\\
   &=r^2-\frac{r'}{1+r}\notag\\
   &=-\frac{tN(3\delta-2)}{2s(1+r)}\,Z.
\end{align}

\subsection{The curvature sign at a stationary point}

Fix a point where $K'=0$ and $\delta>1$. Equation~\eqref{eq:app-Kprime}
and the direct derivative of $K$ imply, respectively,
\begin{equation}
 Z=0,\qquad r'=r^2(1+r).
 \label{eq:app-critical-relations}
\end{equation}
We first differentiate the reciprocal part of $Z$ and then use these
pointwise relations:
\begin{align}
 \left(\frac1u+\frac1v-\frac1r\right)'
   &=-\frac{u'}{u^2}-\frac{v'}{v^2}+\frac{r'}{r^2}\notag\\
   &=-(u-1)-(v-1)+(r+1)\notag\\
   &=3-\delta.
 \label{eq:app-critical-reciprocal}
\end{align}
For $\delta'$, use the ordinary square and cube sums:
\begin{align}
 u'+v'
   &=(u^3+v^3)-(u^2+v^2)\notag\\
   &=(s^3-3st)-(s^2-2t)\notag\\
   &=s^2(s-1)-t(3s-2).
\end{align}
Because $\delta=s-r$, we also have the elementary identities
\[
 s^3-r^3-\delta^3=3sr\delta,\qquad
 s^2+r^2-\delta^2=2sr.
\]
Thus, at the same stationary point,
\begin{align}
 \delta'-\delta^2(\delta-1)
   &=s^2(s-1)-t(3s-2)-r^2(r+1)
     -\delta^2(\delta-1)\notag\\
   &=3sr\delta-2sr-t(3s-2)\notag\\
   &=rt(3\delta-2)Z=0.
 \label{eq:app-critical-deltaprime}
\end{align}
In particular, $\delta'=\delta^2(\delta-1)$ there.

Now differentiate the definition of $Z$ \emph{before} substituting
these stationary-point formulas:
\begin{align}
 Z'
   &=\left(\frac1u+\frac1v-\frac1r\right)'
     +\frac{9\delta'}{(3\delta-2)^2}\notag\\
   &=3-\delta+\frac{9\delta^2(\delta-1)}{(3\delta-2)^2}\notag\\
   &=2+(\delta-1)
     \left[\left(\frac{3\delta}{3\delta-2}\right)^2-1\right]\notag\\
   &>2.
 \label{eq:app-critical-Zprime}
\end{align}
The last step uses only $\delta-1>0$ and
$3\delta/(3\delta-2)>1$.
Since \eqref{eq:app-Kprime} holds on a neighborhood, it can be
differentiated there. At $K'=Z=0$,
\begin{align}
 \frac{K''}{K}
   &=(\log K)''\notag\\
   &=-\frac{uv(r+u)(r+v)(3\delta-2)}
   {2(u+v)(1+r)}\,Z'<0.
 \label{eq:app-curvature-final}
\end{align}
This proves the conditional curvature claim. No equality imposed only at a stationary point
has been differentiated as a neighborhood identity.

\subsection{Checking the condition for both coordinate types}

For two positive lower coordinates, $u>v>1$. Because $r<u$,
\[
 \delta=u+v-r>v>1.
\]

For one positive and one negative coordinate, write
 \begin{equation}
 u>1>v>0,\qquad
 \rho_+=\alpha(u-1),\quad \rho_-=\gamma(1-v),
 \qquad \alpha,\gamma>0.
\end{equation}
We show that the cap on $K$ is exactly the condition $\delta>1$.
From the defining weighted average for $r^2$,
\begin{align}
 \alpha(u-1)(u^2-r^2)
   &=\gamma(1-v)(r^2-v^2),\notag\\
 \frac\gamma\alpha
   &=\frac{(u-r)(u+r)(u-1)}{(r-v)(r+v)(1-v)}.
 \label{eq:app-weight-ratio}
\end{align}
Meanwhile,
\begin{align}
 \gamma-\alpha-K
  &=\frac{(\gamma-\alpha)(1+r)
        -\alpha(u-1)-\gamma(1-v)}{1+r}\notag\\
  &=\frac{\gamma(r+v)-\alpha(r+u)}{1+r}.
\end{align}
Substitute \eqref{eq:app-weight-ratio} to obtain
\begin{equation}
 \gamma-\alpha-K
 =\frac{\alpha(r+u)}{1+r}
  \frac{(u-r)(u-1)-(r-v)(1-v)}{(r-v)(1-v)}.
\end{equation}
The numerator in the second fraction simplifies as follows:
\begin{align*}
 (u-r)(u-1)-(r-v)(1-v)
   &=(u^2-v^2)-(u-v)-r(u-v)\\
   &=(u-v)(u+v-r-1)\\
   &=(u-v)(\delta-1).
\end{align*}
Hence
\begin{equation}
 \gamma-\alpha-K
 =\frac{\alpha(r+u)(u-v)}{(1+r)(r-v)(1-v)}\,(\delta-1).
 \label{eq:app-cap-factor}
\end{equation}
Every factor outside $\delta-1$ is positive. Therefore
\begin{equation}
 K<\gamma-\alpha\quad\Longleftrightarrow\quad \delta>1,
\end{equation}
and the same curvature result applies at every stationary point below
the required cap.

We finish by extracting the interval statement used in the main text.
With two positive lower coefficients, every stationary point of $K$ is a
strict local maximum. Hence a disconnected strict superlevel set would
force an intervening local minimum, a contradiction. In the mixed case,
if $0<\beta<\gamma-\alpha$ and $\{K>\beta\}$ were disconnected, an
intervening local minimum would have value at most $\beta$. By
\eqref{eq:app-cap-factor}, that stationary point lies in the region
$\delta>1$ and is again a strict local maximum, a contradiction. This
proves the two superlevel-set assertions used in Lemma~\ref{lem:curvature}.
 \section{Exact calculations for the sharp example}
\label{app:example}

With common covariance $I_2$, take
\begin{equation}
 \mu_1=(R,0),\quad
 \mu_2=(-R/2,\sqrt3R/2),\quad
 \mu_3=(-R/2,-\sqrt3R/2),\quad
 w_i=\frac13,\quad R^2=\frac83\log2.
\end{equation}
All means have squared norm $R^2$ and distinct means have inner product
$-R^2/2$. At $x=t\mu_1$, the posterior likelihood ratio of the first
component to each other component is
\[
 \exp\bigl(x\cdot(\mu_1-\mu_2)\bigr)
 =\exp(3tR^2/2)=2^{4t}.
\]
Consequently the posterior weights at $t=0,1/4,1/2$ are, respectively,
\[
 (1/3,1/3,1/3),\qquad(1/2,1/4,1/4),\qquad(2/3,1/6,1/6).
\]
Their posterior means are $0$, $\mu_1/4$, and $\mu_1/2$, so all three
points are stationary. Rotation by $2\pi/3$ gives the other points.

For a posterior vector $(p,(1-p)/2,(1-p)/2)$, the mean lies on the first
coordinate axis and the covariance is diagonal. Direct computation gives
\begin{align}
 m_x&=R\left(p-\frac{1-p}{2}\right)=\frac{R(3p-1)}2,\\
 \operatorname{Var}_p(\mu_x)
 &=R^2\left(p+\frac{1-p}{4}\right)-m_x^2
   =\frac{9R^2}{4}p(1-p),\\
 \operatorname{Var}_p(\mu_y)&=\frac{3R^2}{4}(1-p),\qquad
 \operatorname{Cov}_p(\mu_x,\mu_y)=0.
\end{align}
Subtracting the identity gives the two log-density Hessian eigenvalues:
\begin{center}
\renewcommand{\arraystretch}{1.3}
\begin{tabular}{lll}
\toprule
Point & Eigenvalues & Type\\
\midrule
$0$ & $\frac43\log2-1$, $\frac43\log2-1$ & strict mode\\
$\mu_i/2$ & $\frac43\log2-1$, $\frac23\log2-1$ & strict mode\\
$\mu_i/4$ & $\frac32\log2-1$, $\log2-1$ & saddle\\
\bottomrule
\end{tabular}
\end{center}
The bounds $2/3<\log2<3/4$ follow from the strict midpoint and trapezoidal
bounds for the integral of the convex function $1/x$ on $[1,2]$.
They verify all signs in the table without numerical approximation.

The density values at the four modes agree:
\begin{align}
 f(0)&=\frac{e^{-R^2/2}}{2\pi}
       =\frac{2^{-4/3}}{2\pi},\\
 f(\mu_i/2)&=\frac{e^{-R^2/8}+2e^{-7R^2/8}}{6\pi}
       =\frac{2^{-4/3}}{2\pi}.
\end{align}
The saddle value is $f(\mu_i/4)=2^{-3/4}/(3\pi)$. Theorem~\ref{thm:main}
excludes further modes; since the density tends to zero at infinity,
the four equal-height modes are global maxima. Figure~\ref{fig:four}
evaluates this explicit density on a deterministic grid and along a line.
Its near-maximum contour levels emphasize the shallow separation between
these modes. No numerical classification is used in the proof.

\subsection{Computational verification of algebra}
An exact SymPy verification script was used to check the displayed algebra.
The checks cover the moving two-point weights, the variance and
third-moment factorizations, the positive-factor identity for $K'/K$, the stationary
derivative of $Z$, the constrained-Hessian determinant, the fold derivative,
the stationary-height derivative, the unified coordinate derivatives,
the mixed-sign cap factorization, and the last-branch second derivative.
For Appendix~\ref{app:two-on-one}, the script checks the stationary equations,
the nonnegative-quadratic modification and its separable coefficients,
the Hessian forms and branch signs, and the explicit Gaussian realization.
The sign certificates use only $2/3<\log2<3/4$, proved above by integral
bounds. The script differentiates before substituting stationary-point
relations. These checks support the displayed calculations; they do not
replace the global sign arguments, the branch counts, or the cited
finiteness theorem and perturbation arguments.
 \section{Two minima can lie on the same upper branch}\label{app:two-on-one}

At a stationary point of \eqref{eq:potential}, put $h_i=1/p_i-a_i$. We call
coordinate $i$ upper when $h_i<0$ and lower when $h_i>0$. At most one
coordinate can be upper at a strict nondegenerate minimum, but that does
\emph{not} imply uniqueness of the minimum on that branch.

\begin{proposition}\label{counter:prop:counterexample}
There is an acute-triangle Gaussian mixture whose reduced free energy has
two distinct nondegenerate strict minima with coordinate 1 upper and
coordinates 2 and 3 lower.
\end{proposition}

\begin{proof}
\textbf{Preserve two known minima.}
Put $\theta=\log2$ and take the equilateral potential from
Appendix~\ref{app:example}, up to an additive constant:
\begin{equation}\label{counter:eq:base}
 F_0(p)=\sum_i p_i\log p_i-2\theta\sum_i p_i^2,\qquad
 p^{(0)}=\left(\frac13,\frac13,\frac13\right),\quad
 p^{(1)}=\left(\frac23,\frac16,\frac16\right).
\end{equation}
Both displayed points are stationary. For a tangent vector with
$v_1+v_2+v_3=0$, their Hessian forms are
\begin{align*}
 D^2F_0(p^{(0)})[v,v]&=(3-4\theta)\sum_i v_i^2,\\
 D^2F_0(p^{(1)})[v,v]&=\left(\frac92-6\theta\right)v_1^2
                         +(3-2\theta)(v_2-v_3)^2.
\end{align*}
The bounds $2/3<\theta<3/4$, proved in Appendix~\ref{app:example}, make
both forms positive definite. Now set
\begin{equation}\label{counter:eq:penalty}
 \widetilde F(p)=F_0(p)+\frac{\theta}{2}(p_2-p_3)^2.
\end{equation}
At both points $p_2=p_3$, so the added term and its gradient vanish.
Its Hessian form is $\theta(v_2-v_3)^2\geq0$. Thus both points remain
nondegenerate strict minima of $\widetilde F$.

\textbf{Check the branch labels.}
On the simplex, $p_2+p_3=1-p_1$, and therefore
\begin{align*}
 (p_2-p_3)^2
 &=2p_2^2+2p_3^2-(1-p_1)^2\\
 &=-p_1^2+2p_2^2+2p_3^2+2p_1-1.
\end{align*}
Hence $\widetilde F$ has the generic separable form
\[
 \widetilde F(p)=\sum_i\left(
 p_i\log p_i-\frac{a_i}{2}p_i^2-b_ip_i\right)-\frac\theta2,
\]
with
\begin{equation}\label{counter:eq:coefficients}
 (a_1,a_2,a_3)=(5\theta,2\theta,2\theta),\qquad
 (b_1,b_2,b_3)=(-\theta,0,0).
\end{equation}
Write $\mathbf h(p):=(h_1(p),h_2(p),h_3(p))$. The two branch-sign
vectors are
\[
 \mathbf h(p^{(0)})=(3-5\theta,\,3-2\theta,\,3-2\theta),\qquad
 \mathbf h(p^{(1)})=(\tfrac32-5\theta,\,6-2\theta,\,6-2\theta).
\]
Again $2/3<\theta<3/4$ shows that both have signs $(-,+,+)$.
Thus the two minima lie on the same upper branch.

\textbf{Realize the potential by a Gaussian mixture.}
Take covariance $I_2$ and
\begin{equation}\label{counter:eq:mixture}
 \mu_1=(\sqrt{6\theta},0),\quad
 \mu_2=(0,\sqrt\theta),\quad
 \mu_3=(0,-\sqrt\theta),\qquad
 (w_1,w_2,w_3)=\frac{(\sqrt2,1,1)}{2+\sqrt2}.
\end{equation}
The squared distances are $d_{12}^2=d_{13}^2=7\theta$ and $d_{23}^2=4\theta$,
which give the coefficients $a_i$ in \eqref{counter:eq:coefficients}.
All are positive, so the triangle is acute. Moreover,
\[
 \log w_i-\frac{a_i}{2}
 =b_i-\theta-\log(2+\sqrt2),\qquad i=1,2,3.
\]
The mixture potential therefore differs from $\widetilde F$ only by a
constant on the simplex. By Lemma~\ref{lem:appendix-extrema}, its two minima
correspond to distinct modes at
$(\sqrt{6\theta}/3,0)$ and $(2\sqrt{6\theta}/3,0)$.
\end{proof}

The nondegenerate minima and their strict branch signs persist under
sufficiently small perturbations, by the implicit function theorem.
Thus the example is not restricted to the displayed symmetric configuration.
 

\begin{thebibliography}{13}
\providecommand{\natexlab}[1]{#1}
\providecommand{\url}[1]{\texttt{#1}}
\expandafter\ifx\csname urlstyle\endcsname\relax
  \providecommand{\doi}[1]{doi: #1}\else
  \providecommand{\doi}{doi: \begingroup \urlstyle{rm}\Url}\fi

\bibitem[Am{\'e}ndola et~al.(2020)Am{\'e}ndola, Engstr{\"o}m, and
  Haase]{amendola2020}
Carlos Am{\'e}ndola, Alexander Engstr{\"o}m, and Christian Haase.
\newblock Maximum number of modes of {Gaussian} mixtures.
\newblock \emph{Information and Inference: A Journal of the IMA}, 9\penalty0
  (3):\penalty0 587--600, 2020.
\newblock \doi{10.1093/imaiai/iaz013}.

\bibitem[Carreira-Perpi{\~n}{\'a}n(2006)]{carreira2006}
Miguel~{\'A}. Carreira-Perpi{\~n}{\'a}n.
\newblock Acceleration strategies for {Gaussian} mean-shift image segmentation.
\newblock In \emph{2006 {IEEE} Computer Society Conference on Computer Vision
  and Pattern Recognition}, volume~1, pages 1160--1167, 2006.
\newblock \doi{10.1109/CVPR.2006.44}.

\bibitem[Carreira-Perpi{\~n}{\'a}n(2007)]{carreira2007}
Miguel~{\'A}. Carreira-Perpi{\~n}{\'a}n.
\newblock {Gaussian} mean-shift is an {EM} algorithm.
\newblock \emph{{IEEE} Transactions on Pattern Analysis and Machine
  Intelligence}, 29\penalty0 (5):\penalty0 767--776, 2007.
\newblock \doi{10.1109/TPAMI.2007.1057}.

\bibitem[Carreira-Perpi{\~n}{\'a}n and Williams(2003)]{carreirawilliams2003}
Miguel~{\'A}. Carreira-Perpi{\~n}{\'a}n and Christopher K.~I. Williams.
\newblock An isotropic {Gaussian} mixture can have more modes than components.
\newblock Technical Report EDI-INF-RR-0185, School of Informatics, University
  of Edinburgh, 2003.
\newblock URL
  \url{https://homepages.inf.ed.ac.uk/ckiw/postscript/EDI-INF-RR-0185.pdf}.

\bibitem[Ellis and Wang(1990)]{elliswang1990}
Richard~S. Ellis and Kongming Wang.
\newblock Limit theorems for the empirical vector of the {Curie--Weiss--Potts}
  model.
\newblock \emph{Stochastic Processes and their Applications}, 35\penalty0
  (1):\penalty0 59--79, 1990.
\newblock \doi{10.1016/0304-4149(90)90122-9}.

\bibitem[Furman et~al.(1977)Furman, Dattagupta, and Griffiths]{furman1977}
David Furman, Sushanta Dattagupta, and Robert~B. Griffiths.
\newblock Global phase diagram for a three-component model.
\newblock \emph{Physical Review B}, 15\penalty0 (1):\penalty0 441--464, 1977.
\newblock \doi{10.1103/PhysRevB.15.441}.

\bibitem[Gaite et~al.(1999)Gaite, Margalef-Roig, and
  Miret-Art{\'e}s]{gaite1998}
J.~Gaite, J.~Margalef-Roig, and S.~Miret-Art{\'e}s.
\newblock Analysis of a three-component model phase diagram by catastrophe
  theory: Potentials with two order parameters.
\newblock \emph{Physical Review B}, 59:\penalty0 8593--8601, 1999.
\newblock \doi{10.1103/PhysRevB.59.8593}.
\newblock URL \url{https://arxiv.org/abs/cond-mat/9807301}.

\bibitem[Kullback and Leibler(1951)]{kullback1951}
Solomon Kullback and Richard~A. Leibler.
\newblock On information and sufficiency.
\newblock \emph{The Annals of Mathematical Statistics}, 22\penalty0
  (1):\penalty0 79--86, 1951.
\newblock \doi{10.1214/aoms/1177729694}.

\bibitem[K{\"u}lske and Mei{\ss}ner(2020)]{kulske2020}
Christof K{\"u}lske and Daniel Mei{\ss}ner.
\newblock Stable and metastable phases for the {Curie--Weiss--Potts} model in
  vector-valued fields via singularity theory.
\newblock \emph{Journal of Statistical Physics}, 181:\penalty0 968--989, 2020.
\newblock \doi{10.1007/s10955-020-02615-y}.

\bibitem[Neal and Hinton(1998)]{nealhinton1998}
Radford~M. Neal and Geoffrey~E. Hinton.
\newblock A view of the {EM} algorithm that justifies incremental, sparse, and
  other variants.
\newblock In Michael~I. Jordan, editor, \emph{Learning in Graphical Models},
  pages 355--368. Kluwer Academic Publishers, Dordrecht, 1998.
\newblock \doi{10.1007/978-94-011-5014-9_12}.

\bibitem[Okuno and Kabata(2026)]{okuno2026}
Akifumi Okuno and Yutaro Kabata.
\newblock On mixtures of three homoscedastic {Gaussian} densities: An
  unconditional sharper bound on the number of modes, 2026.
\newblock URL \url{https://arxiv.org/abs/2607.17506v1}.
\newblock arXiv:2607.17506v1.

\bibitem[Ray and Lindsay(2005)]{raylindsay2005}
Surajit Ray and Bruce~G. Lindsay.
\newblock The topography of multivariate normal mixtures.
\newblock \emph{The Annals of Statistics}, 33\penalty0 (5):\penalty0
  2042--2065, 2005.
\newblock \doi{10.1214/009053605000000417}.

\bibitem[Shannon(1948)]{shannon1948}
Claude~E. Shannon.
\newblock A mathematical theory of communication.
\newblock \emph{The Bell System Technical Journal}, 27\penalty0 (3):\penalty0
  379--423, 1948.
\newblock \doi{10.1002/j.1538-7305.1948.tb01338.x}.

\end{thebibliography}
\end{document}